\documentclass[reqno]{amsart}
\usepackage{amssymb,amsmath,amsthm,amsfonts,fancyhdr}
\usepackage{indentfirst}
\usepackage{mathrsfs}
\usepackage{setspace}
\usepackage{color}
\usepackage{cite}
\usepackage{bm}
\usepackage{lineno}
\usepackage{latexsym}
\usepackage{verbatim}
\usepackage{cases}
\usepackage{esint}
\theoremstyle{plain} 
\newtheorem{thm}{Theorem}[section]
\newtheorem{pro}[thm]{Proposition}
\newtheorem{lem}[thm]{Lemma}
\newtheorem{cor}[thm]{Corollary}
\theoremstyle{definition}

\theoremstyle{remark}
\newtheorem{rmk}[thm]{Remark}

\renewcommand{\det}{\mbox{det}}

  \numberwithin{equation}{section}
  \numberwithin{figure}{section}

\begin{document}

\title{\textbf{quadratic  growth solutions of Fully nonlinear elliptic equations with periodic data }}

\author{Dongsheng Li}
\author{Lichun Liang}

\address{School of Mathematics and Statistics, Xi'an Jiaotong University, Xi'an, P.R.China 710049.}
\address{School of Mathematics and Statistics, Xi'an Jiaotong University, Xi'an, P.R.China 710049.}

\email{lidsh@mail.xjtu.edu.cn}
\email{lianglichun126@stu.xjtu.edu.cn}

\begin{abstract}
In this paper, we study quadratic growth solutions $u$ of fully nonlinear elliptic equations of the form $F(D^2u)=f$ in $\mathbb{R}^n$, where $f$ is periodic and $F$ may be not uniformly elliptic. The existence of solutions and Liouville type results in the whole space and exterior domains
 are established, which generalize the classical results when $f$ is constant. As applications, the corresponding results are given to $k$-Hessian equations, which include the celebrated results for Monge-Amp\`{e}re equations.
\end{abstract}


\date{}
\maketitle

\section{Introduction}
\noindent

In this paper we are concerned with quadratic  growth solutions of fully nonlinear  elliptic equations of the form
\begin{equation}\label{eq1}
F(D^2u(x))=f(x),\ \ \ x\in \mathbb{R}^n,
\end{equation}
where $f$ is continuous and periodic, i.e.,
$$f(x+z)=f(x)$$
for all $x\in \mathbb{R}^n$ and $z\in \mathbb{Z}^n$.
$F$ is a real valued  continuous  function defined on $\mathcal{S}^{n\times n}$ the space of all real $n\times n$ symmetric matrices and  elliptic, i.e.,
$$F(X)\leq F(Y)$$
whenever $X \leq Y $ for any $X,Y \in \mathcal{S}^{n\times n}$.
 Throughout this paper, when the operator  $F$ is uniformly elliptic, this means that  there are  two  constants $0<\lambda \leq \Lambda < \infty$ such that
$$\lambda \|N\|\leq F(M+N)-F(M)\leq \Lambda\|N\|$$
for any $M, N\in \mathcal{S}^{n\times n}$ and $N\geq 0$.  We also introduce Pucci's extremal operators, i.e.,
$$\mathcal{M}^{+}(M)=\Lambda\sum_{\kappa_i>0} \kappa_i+\lambda \sum_{\kappa_i<0} \kappa_i\ \ \ \mbox{and}\ \ \ \mathcal{M}^{-}(M)=\lambda\sum_{\kappa_i>0} \kappa_i+\Lambda \sum_{\kappa_i<0} \kappa_i,$$
where $\kappa_i=\kappa_i(M) (i=1,\ldots,n)$ are the eigenvalues of $M\in \mathcal{S}^{n\times n}$.

In order to expound  the main motivation on this work, we would like to confine our attention to the development of Liouville type results
for elliptic equations with the periodic data in the whole space.  
In  1989, by implementing the tools from homogenization theory, Avellaneda and Lin \cite{AL} obtained a Liouville type result for linear elliptic equations of divergence form
$\partial_i(a_{ij}(x)\partial_ju(x))=0$ in $\mathbb{R}^n$.
Under the hypothesis that the coefficients $a_{ij}(x)$ are Lipschitz continuous and periodic, they showed that any polynomial growth solution of degree of at most $m$ must be a polynomial with  periodic coefficients.  A few years later, Moser and Struwe \cite{MS92}considered quasilinear elliptic equations $-div (F_{p}(x,Du(x)))=0$ in $\mathbb{R}^n$, where $F(x,p)$ is periodic in $x$ and satisfies convexity and suitable growth assumptions with respect to $p$. Using the Harnack inequality, they showed that any linear growth solution must be a linear function up to a periodic perturbation,
which partially generalizes Avellaneda and Lin's result from the linear to the nonlinear case. Moreover, they also  achieved a simplified proof for the linear case without the Lipschitz continuous assumption on the coefficients.
For linear elliptic equations of non-divergence form $a_{ij}(x)D_{ij}u(x)=0$ in $\mathbb{R}^n$ with measurable and periodic coefficients, Li and Wang \cite{LW}  proved a similar result to \cite{AL}.  For Monge-Amp\`{e}re equations $\det (D^2u(x))=f(x)$ in $\mathbb{R}^n$ with $f$ being periodic, Li \cite{LL90} first established the existence result for entire convex quadratic growth solutions decomposed as a quadratic polynomial plus a periodic function. After ten years,  Caffarelli and Li \cite{CL04}  showed that any convex solution 
must be a quadratic polynomial  up to a periodic perturbation under the condition that $f$ is periodic and smooth. Recently, the smooth assumption on $f$ has been weakened into $f\in L^{\infty}(\mathbb{R}^n)$ by Li and Lu\cite{LL22}.

In light of these invoked results, the aim of this paper is to investigate the existence  and Liouville type results for quadratic growth solutions of (\ref{eq1}). Before stating  our main results, we introduce the space $\mathbb{T}$ consisting of all continuously periodic functions with zero mean, defined by
$$\mathbb{T}=\left\{v\in C(\mathbb{R}^n):v(x+z)=v(x)\ \mbox{for all}\ x\in \mathbb{R}^n\ \mbox{and}\ z\in \mathbb{Z}^n, \fint_{[0,1]^n}v\,dx=0\right\}.$$
\bigskip

\noindent
\textbf{1.1. Uniformly elliptic equations in the whole space.} The following theorem is the cornerstone of the existence result for quadratic  growth solutions of (\ref{eq1}).
\bigskip
\begin{thm}\label{th3}
Let $F\in C^2(\mathcal{S}^{n\times n})$   be  concave (convex) and uniformly elliptic and  $f\in C^\alpha(\mathbb{R}^n)$  for some $0<\alpha<1$ be periodic. Then for any $A\in \mathcal{S}^{n\times n}$,
\begin{equation}\label{eq2}
F(A+D^2v)-\fint_{[0,1]^n}F(A+D^2v)\,dx=f- \fint_{[0,1]^n}f(x)\,dx
\end{equation}
has a unique  solution $v\in C^{2}(\mathbb{R}^n)\cap \mathbb{T}$.
\end{thm}
\bigskip
\begin{rmk}\label{rej1}
Since the term $\fint_{[0,1]^n}F(A+D^2v)$ in (\ref{eq2}) is not a local one, {\it weak solutions} can not be used conveniently, that is the solution $v$ here will be a function of $C^2$. When we conceal this term later, viscosity solution can be adopted and then the smooth assumption on $F$ and $f$ can be relaxed via a smooth approximation and the H\"{o}lder estimate. Please see e.g., the following Corollary \ref{th4} and its proofs.
\end{rmk}

\bigskip
 \begin{cor}\label{th4}
 Let $F\in C(\mathcal{S}^{n\times n})$   be  concave (convex) and uniformly elliptic and  $f\in  C(\mathbb{R}^n)$ be periodic. Then for any $A\in \mathcal{S}^{n\times n}$, there exists a unique $\beta\in \mathbb{R}$ such that
\begin{equation*}
  F(A+D^2v)-\beta=f-\fint_{[0,1]^n}f(x)\,dx
\end{equation*}
has a unique viscosity solution  $v\in \mathbb{T}$.
 \end{cor}
 \bigskip
\begin{rmk}\label{rm1}
Due to the uniqueness of $\beta$ in Corollary \ref{th4}, we define the homogenization operator $\bar{F}: \mathcal{S}^{n\times n} \rightarrow \mathbb{R}$ as
$$A\mapsto \bar{F}(A)=\beta.$$ 
Here we emphasize the dependence of  $\bar F$ on $A$ and actually, it may also depends on $f$.  From Theorem \ref{th3} and Corollary \ref{th4}, we can see that  $\bar{F}(A)=\fint_{[0,1]^n}F(A+D^2v)\,dx$.
 In particular, if $F$ is the trace operator, then $\bar{F}$ is obviously the trace operator itself.
 \end{rmk}
 \bigskip
Consequently, we give the sufficient and necessary condition for the existence of quadratic  growth solutions of (\ref{eq1}).
\bigskip
\begin{thm}\label{th1}
Let $F\in C(\mathcal{S}^{n\times n})$   be  concave (convex) and uniformly elliptic and  $f\in  C(\mathbb{R}^n)$ be periodic. Then for any $A\in \mathcal{S}^{n\times n}$, $b\in \mathbb{R}^n$  and $c\in \mathbb{R}$, (\ref{eq1}) has a viscosity solution $u$ satisfying $u(x)=\frac{1}{2}x^{T}Ax+b\cdot x+c+v(x)$ for some  $v\in \mathbb{T}$ if and only if $\bar{F}(A)=\fint_{[0,1]^n}f(x)\,dx$ . In particular, $v$ is uniquely determined by $A$.
\end{thm}
\bigskip
The following  Liouville theorem indicates that any quadratic  growth solution of (\ref{eq1})  must be a quadratic polynomial plus a periodic function.
\bigskip
 \begin{thm}\label{th2}
 Let $F\in C(\mathcal{S}^{n\times n})$   be    uniformly elliptic and  $f\in  C(\mathbb{R}^n)$ be periodic.
 Assume that $u$ is a viscosity solution of (\ref{eq1}) and satisfies
 \begin{equation}\label{eq36}
   |u(x)|\leq C(1+|x|^2),\ \ \ x\in \mathbb{R}^n
 \end{equation}
 for some constant $C>0$. Then, if either $n\geq 3$ and $F$ is concave (convex) or $n=2$,
 there exist some $A\in \mathcal{S}^{n\times n}$ with $\bar{F}(A)=\fint_{[0,1]^n}f(x)\,dx$, $b\in \mathbb{R}^n$, $c\in \mathbb{R}$ and $v\in \mathbb{T}$ such that $$u(x)=\frac{1}{2}x^{T}Ax+b\cdot x+c+v(x).$$
 \end{thm}
\bigskip
\begin{rmk}\label{rej4}
By our method, the corresponding results can be obtained for the linear elliptic equation of non-divergence form $a_{ij}(x)D_{ij}v(x)=f(x)$ in $\mathbb{R}^n$ with periodic coefficients and right hand terms (cf. the following Remarks \ref{rej2} and \ref{rej3}). Actually, in the proof of Theorem \ref{th3}, the linear equations are considered as using the implicit function theorem.
\end{rmk}
\bigskip
\noindent
\textbf{1.2. Uniformly elliptic equations on exterior domains.} Theorem \ref{th1} enables us to establish the following existence theorem for the Dirichlet problem on exterior domains with prescribed asymptotic behavior at infinity.
\bigskip
\begin{thm}\label{th10}
Let $F\in C(\mathcal{S}^{n\times n})$   be  concave (convex) and uniformly elliptic and $f\in  C(\mathbb{R}^n)$ be periodic. Assume that $A\in \mathcal{S}^{n\times n}$ satisfies $\bar{F}(A)=\fint_{[0,1]^n}f(x)\,dx$ and $\Omega\subset \mathbb{R}^n$ is a domain satisfying a uniform interior sphere condition. Then for any $b\in \mathbb{R}^n$ and $\varphi\in C(\partial \Omega)$, there exist some  $c\in \mathbb{R}$ and $v\in \mathbb{T}$ such that
there exists a  viscosity solution $u\in C(\mathbb{R}^n \backslash \overline{\Omega})$  of
\begin{equation*}
   \left\{
\begin{aligned}
   &F(D^2u)=f\ \ \mbox{in}\ \ \mathbb{R}^n \backslash \overline{\Omega}\\
&u=\varphi\ \  \mbox{on}\ \ \partial \Omega
\end{aligned}
\right.
 \end{equation*}
satisfying
$$\lim_{|x|\rightarrow \infty}\left(u(x)-\frac{1}{2}x^{T}Ax-b\cdot x-c-v(x)\right)=0,$$
where $v\in \mathbb{T}$ is a unique viscosity solution of  $F(A+D^2v)=f$ in $\mathbb{R}^n$. Furthermore, if   $\frac{\Lambda}{\lambda}< n-1$, we have the following estimate
\begin{equation}\label{eq39}
  \left|u(x)-\frac{1}{2}x^{T}Ax-b\cdot x-c-v(x)\right|\leq C|x|^{1-(n-1)\frac{\lambda}{\Lambda}}, \ \ \ x\in \mathbb{R}^n \backslash \overline{\Omega},
\end{equation}
where $C$ is a positive constant.
\end{thm}
\bigskip
\begin{rmk}
\emph{(i).} The domain $\Omega$ is said to satisfy  a uniform interior sphere condition if there is a constant $\rho>0$ such that for any $x_0\in \partial \Omega$ there exists a ball $B_{\rho}(z_{x_0})\subset \Omega$ with $\overline{B_{\rho}(z_{x_0})}\cap \partial \Omega={x_0}$ for some $z_{x_0}\in \Omega$.

\emph{(ii).} Here $|x|^{1-(n-1)\frac{\lambda}{\Lambda}}$ is the fundamental solution of the Pucci's operators (see \cite{L01,ASS11} for more details).
\end{rmk}
\bigskip
As an extension of Liouville theorem (Theorem \ref{th2}), we investigate the asymptotic behavior at infinity of quadratic  growth viscosity solutions  of  (\ref{eq1}) in  exterior domains.
 \bigskip
 \begin{thm}\label{th9}
 Let $F\in C(\mathcal{S}^{n\times n})$ be concave (convex) and uniformly elliptic 
 and $f\in C^\alpha(\mathbb{R}^n)$  for some $0<\alpha<1$ be periodic.
 Assume that $u$ is a viscosity solution of
 $$F(D^2u)=f\ \ \ \mbox{in}\ \ \ \mathbb{R}^n \backslash \overline{B_1}$$
and satisfies
 \begin{equation*}
   |u(x)|\leq C|x|^2,\ \ \ x\in \mathbb{R}^n \backslash \overline{B_1}
 \end{equation*}
 for some constant $C>0$. Then, if   $\frac{\Lambda}{\lambda}< n-1$, there exist some $A\in \mathcal{S}^{n\times n}$ with $\bar{F}(A)=\fint_{[0,1]^n}f(x)\,dx$, $b\in \mathbb{R}^n$, $c\in \mathbb{R}$ and $v\in \mathbb{T}$ such that
 $$\left|u(x)-\frac{1}{2}x^{T}Ax-b\cdot x-c-v(x)\right| \leq C|x|^{1-(n-1)\frac{\lambda}{\Lambda}},\ \ \ x\in \mathbb{R}^n \backslash \overline{B_1}$$
for some cosnstant $C>0$, where $v\in \mathbb{T}$ is a unique viscosity solution of  $F(A+D^2v)=f$ in $\mathbb{R}^n$.
 \end{thm}
\bigskip
\begin{rmk}
When $f$ is a constant, Li, Li and Yuan \cite{LLY20} obtained the asymptotic behavior at infinity under the smooth assumption on $F$. Subsequently, Lian and Zhang \cite{LZP} obtained an asymptotic  result  without the smooth assumption on $F$. As for Monge-Amp\`{e}re equations $\det(D^2u)=1$ in exterior domains, the asymptotic behavior at infinity was established by Caffarelli and Li \cite{CL03}, in which we can also refer to the detailed references for $n=2$. As a generalization of Caffarelli and Li's result, Teixeira and Zhang \cite{TZ16} investigated the case that the right hand is a perturbation of a periodic function.
\end{rmk}
\bigskip
\noindent
\textbf{1.3. Degenerate elliptic equations in the whole space.}
For degenerate elliptic operator $F$, we assume that $F$ is  concave on some open convex  set $\Gamma\subset \mathcal{S}^{n\times n}$. Let $ \widetilde{\Gamma}$ be an open convex subset of $\Gamma$.

The existence and Liouville type results will be generalised to  degenerate elliptic equations. For this purpose, we suppose:

\textbf{(H1):} Given  $A\in \widetilde{\Gamma}$ and $f\in C^\alpha(\mathbb{R}^n)\cap \mathbb{T}$ $(0<\alpha<1)$, for any $0\leq t\leq 1$, if $v_t\in C^{2}(\mathbb{R}^n)\cap \mathbb{T}$  with $A+D^2v_t\in \Gamma$  is a   solution of
\begin{equation*}
  F(A+D^2v_t)-\fint_{[0,1]^n}F(A+D^2v_t)\,dx=tf,
\end{equation*}
then $A+D^2v_t$ for all $0\leq t\leq 1$ lie in a compact  set $K$ of $\Gamma$.

\textbf{(H2):} $F$ is uniformly elliptic with respect to $u$ in the sense that if $D^2u$ lies in a compact set $K$ of $\Gamma$,
then there exists a constant $0<\lambda_K \leq 1$ depending only on $K$ such that
$$\lambda_K |\xi|^2\leq F_{ij}(D^2u)\xi_i \xi_j\leq \frac{1}{\lambda_K}|\xi|^2$$
for all $\xi\in \mathbb{R}$.
\bigskip
\begin{thm}\label{th11}
Let $F\in C^2(\mathcal{S}^{n\times n})$   satisfy (H1) and (H2)
 and  $f\in C^\alpha(\mathbb{R}^n)\cap \mathbb{T}$  for some $0<\alpha<1$. Then for any $A\in \widetilde{\Gamma}$,
\begin{equation*}
F(A+D^2v)-\fint_{[0,1]^n}F(A+D^2v)\,dx=f
\end{equation*}
has a unique  solution $v\in C^{2}(\mathbb{R}^n)\cap \mathbb{T}$ with $A+D^2v\in \Gamma$.
\end{thm}
\bigskip

On account of  Remark \ref{rm1}, we can similarly define the  homogenization operator $\bar{F}: \widetilde{\Gamma} \rightarrow \mathbb{R}$  for degenerate elliptic operators.
In order to obtain  the Liouville type result for degenerate elliptic equations with periodicity, we need  suppose:


\textbf{(H3):} The homogenization operator $\bar{F}$ possesses a Liouville property, i.e., if $u$ is  a quadratic  growth viscosity solution of $\bar{F}(D^2u)=c$ in $\mathbb{R}^n$  for some $c\in\mathbb{R}$ with $D^2u\in \widetilde{\Gamma}$ in the weak sense, then $u$ must be a quadratic polynomial.  We say that $D^2u\in \widetilde{\Gamma}$ in the weak sense if $u-\varphi$ has a local maximum at $x_0\in \mathbb{R}^n$ with $\varphi \in C^2(\mathbb{R}^n)$, then we have
$D^2\varphi(x_0) \in \widetilde{\Gamma}$.
\bigskip
\begin{thm}\label{th8}
Let $F\in C^2(\mathcal{S}^{n\times n})$ satisfy (H1)-(H3) . Assume that $u\in C^2(\mathbb{R}^n)$ with $D^2u\in \widetilde{\Gamma}$ is a solution of
$$F(D^2u)=f\ \ \ \mbox{in}\ \ \ \mathbb{R}^n$$
for some periodic function $f\in C^\alpha(\mathbb{R}^n)$  $(0<\alpha<1)$. Then, if $D^2u$ is bounded in $\mathbb{R}^n$, there exist some $A\in \mathcal{S}^{n\times n}$ with $\bar{F}(A)=\fint_{[0,1]^n}f(x)\,dx$, $b\in \mathbb{R}^n$, $c\in \mathbb{R}$ and $v\in \mathbb{T}$ such that $$u(x)=\frac{1}{2}x^{T}Ax+b\cdot x+c+v(x).$$
In particular, $v$ is uniquely determined by $A$.
\end{thm}
\bigskip
\noindent
\textbf{1.4. Applications to  $k$-Hessian equations.}
We consider $k$-Hessian equations as an interesting example of Theorem \ref{th11}. Following \cite{CNS85} and \cite{Iv83}, we recall that the cone $\Gamma_k$ is given by
$$\Gamma_k=\{\kappa\in \mathbb{R}^n: \sigma_j(\kappa)>0, \ \ j=1,2,\ldots, k\},$$
where $$\sigma_k(\kappa)=\sum_{i_1<\cdots< i_k}\kappa_{i_1}\cdots \kappa_{i_k}$$
is the $k$-th elementary symmetric polynomial.  Let $\kappa(A)=(\kappa_1,\ldots,\kappa_n)$ be the eigenvalues of a matrix $A$, then $\sigma_k(A)=\sigma_k(\kappa(A))$.

We  prove the following:
\bigskip
\begin{thm}\label{th6}
Let $f\in C^{2}(\mathbb{R}^n)$  be  periodic.
For any positive definite matrix $A$ satisfying  $f-\fint_{[0,1]^n}f(x)\,dx+\sigma_k(A)>0$ in $\mathbb{R}^n$,
\begin{equation}\label{eq33}
  \sigma_k(A+D^2v)=f-\fint_{[0,1]^n}f(x)\,dx+\sigma_k(A)
\end{equation}
has a  unique  solution $v\in C^{2}(\mathbb{R}^n)\cap \mathbb{T}$ with $\kappa(A+D^2v)\in \Gamma_k$.
\end{thm}
\bigskip
\begin{cor}
Let $f\in C^{2}(\mathbb{R}^n)$  be  periodic and positive. Assume that $A\in \mathcal{S}^{n\times n}$ is positive definite and satisfies $\sigma_k(A)=\fint_{[0,1]^n}f(x)\,dx$. Then for any $b\in \mathbb{R}^n$  and $c\in \mathbb{R}$,
$$\sigma_k(D^2u)=f$$
has a solution $u\in C^{2}(\mathbb{R}^n)$ with $D^2u \in \Gamma_k$ and  $u(x)=\frac{1}{2}x^{T}Ax+b\cdot x+c+v(x)$ for some  $v\in \mathbb{T}$ uniquely determined by $A$.
\end{cor}
\bigskip
\begin{rmk}\label{rm2}
It is worth pointing out that  Li  \cite{LL90} first obtained the existence of periodic solutions $v$  of $\sigma_k(A+D^2v)=f(x,v)$ with positive definite matrix $A$ in $\mathbb{R}^n$ under the condition that $f(x,v)$ is periodic in $x$ and satisfies $\max_{x\in \mathbb{R}^n}f(x,a)\leq \sigma_k(A) \leq \min_{x\in \mathbb{R}^n}f(x,b)$ for two constants $a<b$. In particular, when $k=n$, he obtained the existence  of  entire convex quadratic  growth solutions $u$ of $\det (D^2u(x))=f(x)$ with $f$ being periodic.  As a result, Theorem \ref{th6} is a generalization of  Monge-Amp\`{e}re equations case. From Theorem \ref{th6}, it is clear that for the $k$-Hessian operator, the homogenization operator $\bar{F}$ is the $k$-Hessian operator itself.  Theorem \ref{th8} can immediately be applied to Monge-Amp\`{e}re equations by setting $\Gamma=\widetilde{\Gamma}=\{A\in \mathcal{S}^{n\times n}: \kappa(A)\in \Gamma_n\}$ and using the boundness of $D^2u$ obtained by Caffarelli and Li \cite{CL04}.
\end{rmk}
\bigskip
As another interesting application of Theorem \ref{th8}, we give a  Liouville type result for  $\sigma_2$-equations.
\bigskip
\begin{thm}\label{th7}
Let $f\in C^2(\mathbb{R}^n)$ be periodic and positive. Assume $u\in C^2(\mathbb{R}^n)$ with $D^2u>0$ is a  solution of
$$\sigma_2(D^2u(x))=f(x),\ \ \ x\in \mathbb{R}^n. $$
If there exist some  $C>0$ and $p>2n$ such that
\begin{equation}\label{eq44}
  \int_{B_r}|D^2u|^{p}\,dx \leq Cr^n,\ \ \ r>1,
\end{equation}
then there exist some $A\in \mathcal{S}^{n\times n}$ with $\sigma_2(A)=\fint_{[0,1]^n}f(x)\,dx$, $b\in \mathbb{R}^n$, $c\in \mathbb{R}$ and $v\in \mathbb{T}$ such that $$u(x)=\frac{1}{2}x^{T}Ax+b\cdot x+c+v(x).$$
\end{thm}
\bigskip
\begin{rmk}
Convex solutions for  $\sigma_2$-equations are considered for the reason that there exists a non-polynomial entire 2-convex solution of $\sigma_2(D^2u)=1$, which was constructed by Warren \cite{WM16}. When $f=1$,  Chang and Yuan \cite{CY10} proved that any entire almost convex smooth solution $u$ of $\sigma_2(D^2u)=1$ must be a quadratic function. Recently, this result has be generalised to entire  convex viscosity solutions by Shankar and Yuan \cite{SY21}. Moreover, Mooney \cite{MC21} gave a  different proof. Finally,  it is worth pointing out that the appeared integral condition is motivated by \cite{BCGJ}.
\end{rmk}
\bigskip
This paper is organized as follows. In Section 2, we use the method of continuity to prove Theorem \ref{th3} and  obtain Theorem \ref{th4} by a smooth approximation. Furthermore, with the help of the uniform ellipticity  for $\bar{F}$, we  achieve the proof  of Theorem \ref{th1}. After determining the second order term via the blow-down for solutions, we  can reduce the proof of Theorem \ref{th2} to analogous analysis to the Monge-Amp\`{e}re equation \cite{CL04}. In Section 3, we use Theorems \ref{th1} and \ref{th2} to prove
Theorems \ref{th10}  and \ref{th9}. In Section 4, applying the preceding arguments, we generalise the results in Theorems \ref{th1} and \ref{th2} to  degenerate elliptic equations, that is Theorems \ref{th11} and \ref{th8}. Finally, in Section 5, considering applications of Theorems \ref{th11} and \ref{th8} to $k$-Hessian equations, we prove Theorems \ref{th6} and \ref{th7}.
\bigskip

We close this section by introducing some notations.

$\bullet$ For $r>0$, $Q_r=[-\frac{r}{2},\frac{r}{2}]^n$.

$\bullet$ For $x\in \mathbb{R}^n$ and $r>0$, $B_r(x)=\{y\in \mathbb{R}^n: |x-y|<r\}$ and $B_r=B_r(0)$.

$\bullet$ $e_1=(1,0,\ldots, 0),\ldots, e_n=(0,0,\ldots, 1)$.

$\bullet$   Set $$E=\{k_1 e_1+\cdots +k_n e_n: k_1,\ldots, k_n\ \ \mbox{are integers}\},$$
and the second order difference quotients
$$\Delta_{h}^2u(x)=\frac{u(x+h)+u(x-h)-2u(x)}{\|h\|^2}\ \ \mbox{for}\ \ h\in \mathbb{R}^n.$$

$\bullet$  The  H\"{o}lder seminorm
$$[u]_{\alpha; \Omega}=\sup_{x,y \in \Omega}\frac{|u(x)-u(y)|}{|x-y|^\alpha}.$$
\bigskip
\section{Uniformly elliptic equations in the whole space}
\noindent

In this section, we establish the existence and Liouville type results for quadratic  growth solutions of uniformly elliptic equations with the periodic data in the whole space. Throughout this section, we set $X=C^{2,\alpha}(\mathbb{R}^n) \cap \mathbb{T}$ and $Y=C^{ \alpha}(\mathbb{R}^n) \cap \mathbb{T}$ for some $0<\alpha<1$.
In addition, we  assume $F(0)=0$.
\bigskip
\subsection{Existence}
\noindent

Firstly, we will prove Theorem \ref{th3} by using the method of continuity. Secondly, we will complete the proof of Corollary \ref{th4} by a smooth approximation of $F$ and $f$ and the H\"{o}lder estimate.
\bigskip
\begin{proof}[Proof of Theorem \ref{th3}]
\emph{Uniqueness.} We let $v, w\in C^{2}(\mathbb{R}^n)\cap \mathbb{T}$ be  solutions of (\ref{eq2}). It follows that
either
$$F(A+D^2v)\leq F(A+D^2w)$$
or
$$F(A+D^2w)\leq F(A+D^2v).$$
Clearly, by the strong maximum principle and $v, w\in \mathbb{T}$, we obtain  $v=w$.

\emph{Existence.}
The proof mainly relies  on the method of continuity. Without loss of generality we  assume $A=0$, $\fint_{Q_1}f(x)\,dx=0$ and $f\in C^{\alpha}(\mathbb{R}^n)$ for some small enough $\alpha>0$.
Now we  consider the  map $ \mathcal{F} :   X \times [0,1]\longrightarrow Y$ defined by
 $$(v,t)\longmapsto \mathcal{F}(v,t)= F(D^2v)-\fint_{Q_1}F(D^2v)\,dx-tf,$$
and the set
$$\mathcal{T}:=\{t\in [0,1]: \mathcal{F}(v_t,t)=0 \ \ \mbox{for some}\ \ v_t\in X\}.$$
It is obvious that $0\in \mathcal{T}$ since $v_0=0$  is a unique solution. So to establish the existence of solutions $v\in X$  of (\ref{eq2}), i.e., $1\in \mathcal{T}$, it suffices to show that $\mathcal{T}$ is both open and closed in $[0,1]$.

\emph{Step 1}: \textbf{$\mathcal{T}$ is closed.} Let $v_t\in X$ denote a solution of $ \mathcal{F}(v_t,t)=0$.
 From the Evans-Krylov theory (see \cite[Theorem 8.1]{CC95}), we have
$$ \|v_t\|_{C^{2,\alpha} (Q_1)} \leq C \left(\|v_t\|_{L^{\infty}(Q_2)}+\|f\|_{C^{\alpha}(Q_2)}+\left|\fint_{Q_1}F(D^2v_t)\,dx\right|\right),$$
where  $C>0$  depends only on $n$, $\lambda$, $\Lambda$ and $\alpha$.
By the aid of the uniform ellipticity condition for $F$ and the interpolation inequality,
we arrive at the estimate
\begin{equation*}
  \|v_t\|_{C^{2,\alpha} (Q_1)}\leq C (\|v_t\|_{L^{\infty}(Q_1)}+\|f\|_{C^{\alpha}(Q_1)}).
\end{equation*}

Next we need to show that
there exists a constant $C>0$ depending only on $\lambda$, $\Lambda$ and  $n$ such that
\begin{equation}\label{eq4}
  \|v_t\|_{L^{\infty}(Q_1)}\leq C \|f\|_{L^{\infty}(Q_1)}.
\end{equation}
To see this, since $v_t$ has a local minimum at $z\in Q_1$ and a local maximum  at $y\in Q_1$, by the uniform ellipticity for $F$, we obtain
$$    -tf(z)\leq \fint_{Q_1}F(D^2v_t)\,dx \leq  -tf(y),$$
which implies that
$$\left|\fint_{Q_1}F(D^2v_t)\,dx\right|\leq \|tf\|_{L^{\infty}(Q_1)}\leq\|f\|_{L^{\infty}(Q_1)}.$$
Combining this with the  Harnack inequality, we have
$$\|v_t-\min_{Q_1}v_t\|_{L^{\infty}(Q_1)}\leq C\|f\|_{L^{\infty}(Q_1)},$$
where $C>0$ depends only on $\lambda$, $\Lambda$ and  $n$. Consequently, we establish the estimate (\ref{eq4}) since $v_t\in X$.
Finally,  we obtain a priori estimates
\begin{equation*}
  \|v_t\|_{C^{2,\alpha} (Q_1)}\leq C \|f\|_{C^{\alpha}(Q_1)}
\end{equation*}
for some constant $C>0$ depending only on $n$, $\lambda$, $\Lambda$ and $\alpha$, from which the closeness of $\mathcal{T}$ follows clearly.

 \emph{Step 2}: \textbf{$\mathcal{T}$ is open.} Openness follows from the implicit function theorem in Banach spaces. It is easy to check that $\mathcal{F}$ is of class $C^1$ and the Frech\`{e}t differential of $\mathcal{F}$ at $(v,t)$ with respect to $v$ 
 is given by
 $$h\mapsto D_v\mathcal{F}(v,t)[h]=F_{ij}(D^2v)D_{ij}h-\fint_{Q_1}F_{ij}(D^2v)D_{ij}h\,dx.$$
It remains to see that $D_v\mathcal{F}(v,t)$ is a linear isomorphism between $X$ and $Y$.  More precisely, we need to establish the existence and  uniqueness of solutions $h\in X$ of
\begin{equation}\label{eq3}
  Lh:=a_{ij}D_{ij}h-\fint_{Q_1}a_{ij}D_{ij}h\,dx=f
\end{equation}
for any $f\in Y$, where the  coefficients $a_{ij}$ are periodic and of class $C^{ \alpha} (0<\alpha<1)$ satisfying
\begin{equation}\label{eq23}
  \bar{\lambda} |\xi|^2\leq a_{ij}(x) \xi_i \xi_j\leq \bar{\Lambda} |\xi|^2
\end{equation}
for any $x, \xi \in \mathbb{R}^n$, where $0<\bar{\lambda}\leq \bar{\Lambda} < \infty$.

In order to apply  the method of continuity  to solve the problem (\ref{eq3}). We consider the family of equations
$$L_{t}h:=(1-t) L_0h+tLh=f,\ \ \ 0\leq t \leq 1,$$
where $L_0h=\Delta h-\fint_{[0,1]^n}\Delta h\,dx=\Delta h$.
We note that   the coefficients of $L_t$ satisfy (\ref{eq23}) with  $\bar{\lambda}$ and $\bar{\Lambda}$ replaced by $$\bar{\lambda}_t=\min(1,\bar{\lambda}) \ \ \mbox{and}\ \ \bar{\Lambda}_t=\max(1,\bar{\Lambda}).$$
 Let $h_t\in X$ denote a solution of this problem. By virtue of the Schauder estimate and the interpolation inequality, we have
\begin{equation*}
\|h_t\|_{C^{2,\alpha} (Q_1)} \leq C (\|h_t\|_{L^{\infty}(Q_1)}+\|f\|_{C^{\alpha}(Q_1)}),
\end{equation*}
where $C>0$ depends only on $\bar{\lambda}$, $\bar{\Lambda}$, $\alpha$, $n$ and $\|a_{ij}\|_{C^{\alpha}(Q_1)}$. Following the same arguments to derive the estimate (\ref{eq4}), we obtain
\begin{equation*}
  \|h_t\|_{L^{\infty}(Q_1)} \leq C \|f\|_{L^{\infty}(Q_1)}
\end{equation*}
for some constant $C>0$ depending only on $\bar{\lambda}$, $\bar{\Lambda}$ and  $n$.

It follows that
$$\|h_t\|_{C^{2,\alpha} (Q_1)} \leq C \|f\|_{C^{\alpha}(Q_1)},$$
that is,
$$\|h_t\|_{C^{2,\alpha} (Q_1)} \leq C \|L_th_t\|_{C^{\alpha}(Q_1)}.$$
Since $L_0=\Delta$ maps $X$ onto $Y$, the method of continuity is applicable. Consequently, the problem (\ref{eq3}) has a solution $h\in X$ for any $f\in Y$ and the uniqueness follows from the strong maximum principle.
 \end{proof}
\bigskip
By the aid of  Theorem \ref{th3}, we prove Corollary \ref{th4}.
\bigskip
\begin{proof}[Proof of Corollary \ref{th4}]
Without loss of generality we assume $A=0$ and $\fint_{[0,1]^n}f(x)\,dx=0$.
For $\epsilon >0$, let $F_\epsilon$ be the mollification of $F$ in $\mathcal{S}^{n\times n}$ and $f_\epsilon$ be the mollification of $f$ in $\mathbb{R}^n$. From  Theorem \ref{th3},
it follows that there is a unique solution $v_\epsilon\in Y$ such that
\begin{equation}\label{eq6}
  F_\epsilon(D^2v_\epsilon)-\fint_{Q_1}F_\epsilon(D^2v_\varepsilon)\,dx=f_\epsilon.
\end{equation}
Clearly, by $F(0)=0$, we note  that  $|F_\epsilon(0)|\leq 1$ for small enough $\epsilon>0$. Furthermore, since $D^2v_\epsilon(y)\leq0$ as $v_\epsilon(y)=\max_{Q_1}v_\epsilon$ and $D^2v_\epsilon(z)\leq0$ as $v_\epsilon(z)=\min_{Q_1}v_\epsilon$,
 we  know that
$$-1-\sup_{Q_1}f_\epsilon \leq   F_\epsilon(0)-f_\epsilon(z)\leq \fint_{Q_1}F_\epsilon(D^2v_\epsilon)\,dx \leq  F_\epsilon(0)-f_\epsilon(y)\leq 1
-\inf_{Q_1}f_\epsilon,$$
that is,
$$\left|\beta_\epsilon:=\fint_{Q_1}F_\epsilon(D^2v_\epsilon)\,dx\right|\leq 1+\|f_\epsilon\|_{L^{\infty}(Q_1)}\leq 1+\|f\|_{L^{\infty}(Q_1)}.$$
Combining this  with  the Harnack inequality and the interior H\"{o}lder estimate, we have
$$\|\bar{v}_\epsilon:=v_\epsilon-\min_{Q_1}v_\epsilon\|_{C^\alpha(Q_1)}\leq C(1+\|f\|_{L^{\infty}(Q_1)}),$$
where $0<\alpha <1$ and $C>0$ depends only on $\lambda$, $\Lambda$ and  $n$. Then we conclude that there exists a periodic function $\bar{v}$ and a real number $\beta$ such that, up to a subsequence,
$$\bar{v}_\epsilon \rightarrow \bar{v}\ \ \ \mbox{in}\ \ \ C(Q_1)\ \ \ \mbox{and}\ \ \ \ \beta_\epsilon \rightarrow \beta \ \ \mbox{as}\ \ \epsilon\rightarrow 0.$$
Hence, letting $\epsilon\rightarrow 0$ in (\ref{eq6}), we obtain
\begin{equation}\label{eq7}
  F(D^2v)-\beta=f \ \ \ \mbox{in}\ \ \ \mathbb{R}^n
\end{equation}
in the viscosity sense, where $v=\bar{v}-\fint_{Q_1}\bar{v}\,dx.$

It remains to show the uniqueness of $\beta$. Indeed, suppose that $w\in \mathbb{T}$ and $\tilde{\beta}$ satisfy (\ref{eq7}).
From the strong maximum principle for viscosity solutions and $v, w\in \mathbb{T}$, it follows that $v=w$. Consequently, $\beta=\tilde{\beta}$.
\end{proof}
\bigskip
 \begin{proof}[Proof of Theorem \ref{th1}]
\emph{ Necessity.} Conversely, suppose that $\bar{F}(A)\neq\fint_{Q_1}f(x)\,dx$. Using Corollary \ref{th4}, we let $w\in \mathbb{T}$ be a viscosity solution
of $ F(A+D^2w)=f-\fint_{Q_1}f(x)\,dx+\bar{F}(A)$.  Clearly, by the strong maximum principle  for viscosity solutions and $v, w\in \mathbb{T}$, we obtain  $v=w$, which yields  $\bar{F}(A)=\fint_{Q_1}f(x)\,dx$, a contradiction.

\emph{ Sufficiency.} It is clear from Corollary \ref{th4}.
 \end{proof}
 \bigskip
 \begin{rmk}\label{rej2}
 We can give the explicit homogenization operator for  the linear elliptic equation of non-divergence form. Let the periodic function $m$ be the unique solution of the problem $D_{ij}(a_{ij}(x)m(x))=0$ in $\mathbb{R}^n$  with $\fint_{Q_1}m(x)\,dx=1$ (see \cite[Theorem 2]{AL89} and \cite{BLP78}). Multiplying the function $m$ to both sides of the equation $a_{ij}(A_{ij}+D_{ij}v)-\fint_{Q_1}a_{ij}(x)(A_{ij}+D_{ij}v(x))\,dx=f-\fint_{Q_1}f(x)\,dx$ and then integrating over $Q_1$, we obtain $\bar{F}(A):=\fint_{Q_1}a_{ij}(x)(A_{ij}+D_{ij}v(x))\,dx=A_{ij}\fint_{Q_1}a_{ij}(x)m(x)\,dx-\fint_{Q_1}f(x)m(x)\,dx+\fint_{Q_1}f(x)\,dx$.
 Clearly, the existence result holds if and only if  $A_{ij}\fint_{Q_1}a_{ij}(x)m(x)\,dx=\fint_{Q_1}f(x)m(x)\,dx$.
 \end{rmk}

\bigskip
At the end of this section, we   collect some properties of the homogenisation operator $\bar{F}$, which come from \cite{Ev92}.
 \bigskip
 \begin{lem}\label{le8}
 (i) If the  operator $F$ is uniformly elliptic, so is the homogenization operator $\bar{F}$ with the same ellipticity constants as the operator $F$.

 (ii) If the operator $F$ is concave in $\mathcal{S}^{n\times n}$, so is the homogenization operator $\bar{F}$ in $\mathcal{S}^{n\times n}$.
 \end{lem}
 \begin{proof}
 (i) To obtain a contradiction, suppose that there exist some $M, N\in \mathcal{S}^{n\times n}$ and $N\geq 0$ such that
 \begin{equation}\label{eq9}
\bar{F}(M+N)-\bar{F}(M) <\lambda \|N\|.
\end{equation}
Let $v^{M}$, $v^{M+N}\in \mathbb{T}$ be  viscosity solutions of
 \begin{equation}\label{eq11}
   \left\{
\begin{aligned}
   &F(D^2v^{M}(x)+M)=f(x)-\fint_{Q_1}f(x)\,dx+\bar{F}(M),\\
&F(D^2v^{M+N}(x)+M+N)=f(x)-\fint_{Q_1}f(x)\,dx+\bar{F}(M+N).
\end{aligned}
\right.
 \end{equation}
We now claim that
\begin{equation}\label{eq10}
  F(D^2v^{M+N}(x)+M) <F(D^2v^{M}(x)+M)
\end{equation}
in the viscosity sense.
To see this, let $\phi\in C^2(\mathbb{R}^n)$ and $v^{M+N}-\phi$ has a local minimum at a point $x_0\in \mathbb{R}^n$. In view of (\ref{eq9}), (\ref{eq11}) and the uniform ellipticity of $F$, we have
\begin{equation*}
  \begin{split}
    F(D^2\phi(x_0)+M) & \leq F(D^2\phi(x_0)+M+N)-\lambda \|N\| \\
      & \leq f(x_0) -\fint_{Q_1}f(x)\,dx+\bar{F}(M+N)-\lambda \|N\|\\
      &<f(x_0)-\fint_{Q_1}f(x)\,dx+\bar{F}(M)\\
      &=F(D^2v^{M}(x_0)+M),
  \end{split}
\end{equation*}
which establishes (\ref{eq10}).
Owing to (\ref{eq10}) and the  strong maximum principle for viscosity solutions, we discover
$$v^{M}-v^{M+N}=c$$
for some constant $c$, a contradiction to (\ref{eq10}). It follows that
$$\lambda \|N\| \leq \bar{F}(M+N)-\bar{F}(M)$$
for any $M, N\in \mathcal{S}^{n\times n}$ and $N\geq 0$.
The same argument works for
$$ \bar{F}(M+N)-\bar{F}(M)\leq \Lambda \|N\|.$$

(ii) For later contradiction, let us suppose that there exist some $M, N\in \mathcal{S}^{n\times n}$ such that
$$\bar{F}\left(\frac{M+N}{2}\right)<\frac{1}{2}\bar{F}(M)+\frac{1}{2}\bar{F}(N).$$
 Let $v^{M}$, $v^{N}$, $v^{\frac{M+N}{2}}\in \mathbb{T}$ be  viscosity solutions of
 \begin{equation*}
   \left\{
\begin{aligned}
   &F(D^2v^{M}(x)+M)=f(x)-\fint_{Q_1}f(x)\,dx+\bar{F}(M),\\
&F(D^2v^{N}(x)+N)=f(x)-\fint_{Q_1}f(x)\,dx+\bar{F}(N),\\
&F\left(D^2v^{\frac{M+N}{2}}(x)+\frac{M+N}{2}\right)=f(x)-\fint_{Q_1}f(x)\,dx+\bar{F}\left(\frac{M+N}{2}\right).
\end{aligned}
\right.
 \end{equation*}
 For $\varepsilon>0$, let $v_{\varepsilon}^{M}$ be the mollification of $v^{M}$ in $\mathbb{R}^n$ and $v_{\varepsilon}^{N}$ be the mollification of $v^{N}$ in $\mathbb{R}^n$. Then we have,  in view of concavity of $F$,

 $$F\left(\frac{D^2v_{\varepsilon}^{M}(x)+D^2v_{\varepsilon}^{N}(x)}{2}+\frac{M+N}{2}\right) \geq \frac{1}{2}F\left(D^2v_{\varepsilon}^{M}(x)+M\right)+\frac{1}{2}F\left(D^2v_{\varepsilon}^{N}(x)+N\right).$$
Sending  $\varepsilon$ to zero, we obtain

\begin{equation*}
  \begin{split}
    F\left(\frac{D^2v^{M}(x)+D^2v^{N}(x)}{2}+\frac{M+N}{2}\right) & \geq \frac{1}{2}F\left(D^2v^{M}(x)+M\right)+\frac{1}{2}F\left(D^2v^{N}(x)+N\right)\\
      & =\frac{1}{2}\bar{F}(M)+\frac{1}{2}\bar{F}(N)+f(x)-\fint_{Q_1}f(x)\,dx\\
      &>\bar{F}\left(\frac{M+N}{2}\right)+f(x)-\fint_{Q_1}f(x)\,dx\\
      &=F\left(D^2v^{\frac{M+N}{2}}(x)+\frac{M+N}{2}\right)
  \end{split}
\end{equation*}
in   the viscosity sense, which together with the strong maximum principle easily yields
$$v^{\frac{M+N}{2}}-\frac{v^{M}+v^{N}}{2}=c$$
for some constant $c$. But we obtain a contradiction.
 \end{proof}
 \bigskip
\begin{rmk}
\emph{(i).}  From (i) in Lemma \ref{le8}, it follows that there exists some $t\in \mathbb{R}$ such that $\bar{F}(tI)=\fint_{Q_1}f(x)\,dx$. Hence we  can always find some $A\in \mathcal{S}^{n\times n}$ satisfying  $\bar{F}(A)=\fint_{Q_1}f(x)\,dx$.

\emph{(ii).}
In the property (ii), if $F$ is concave in some open convex  set $\Gamma \subset \mathcal{S}^{n\times n}$, so is $\bar{F}$ in $\Gamma$.
\end{rmk}
\bigskip
\subsection{ Liouville type result}
\noindent

This subsection will be devoted to the proof of Theorem \ref{th2}.  For the case $n\geq 3$, following the strategy  implemented  by Caffarelli and Li \cite{CL04}, we will divide the proof into a sequence of lemmas. For the case $n=2$, a simpler proof will be presented.

Throughout  this subsection,  $u$ is a viscosity solution of
$$F(D^2u)=f\ \ \ \mbox{in} \ \ \ \mathbb{R}^n$$
with quadratic growth (\ref{eq36}), where $F$ is concave and uniformly elliptic and $f\in C(\mathbb{R}^n)$ is periodic.

 For $R\geq 1$, let
 $$u_R(x)=\frac{u(R x)}{R^2},\ \ \ x\in \mathbb{R}^n.$$
  \bigskip
 \begin{lem}\label{le5}
There exists some  $A\in \mathcal{S}^{n\times n}$ such that a subsequence
$\{u_{R_i}\}_{i=1}^\infty$ converges uniformly in the $C^1$ norm  to $Q(x):=\frac{1}{2}x^TAx$ in any compact set of $\mathbb{R}^n$
with $$\bar{F}(A)=\fint_{Q_1}f(x)\,dx.$$
\end{lem}
\begin{proof}
\emph{Step 1: $u_{R_i} \rightarrow Q\  \ \mbox{in}\  \ C^{1}_{loc}(\mathbb{R}^n)$.}
 Applying the $C^{1,\alpha}$ estimate $(0<\alpha<1)$ (see \cite[Theorem 8.3]{CC95}) to
$$F(D^2u_R(x))=f(Rx),\ \ \ x\in B_r \ (r>1),$$
we obtain
\begin{equation*}
  \begin{split}
   \|u_R\|_{L^{\infty}(B_r)} +r\|Du_R\|_{L^{\infty}(B_r)}+r^{1+\alpha}[Du_R]_{\alpha; B_r}& \leq C (\|u_R\|_{L^{\infty}(B_{2r})}+4r^2\|f(Rx)\|_{L^{\infty}(B_{2r})})\\
      & \leq C (1+4r^2+4r^2\|f\|_{L^{\infty}(\mathbb{R}^n)})
  \end{split}
\end{equation*}
for some constant $C>0$, 
 which implies that we  extract a subsequence $\{u_{R_i}\}_{i=1}^{\infty}$  such that
$$u_{R_i} \rightarrow Q\ \ \ \mbox{in}\ \ \ C^{1}_{loc}(\mathbb{R}^n)$$
with $$|Q(x)|\leq C(1+|x|^2),\ \ \ x\in \mathbb{R}^n$$
for some constant $C>0$.

\emph{Step 2: $Q(x)=\frac{1}{2}x^{T}Ax.$} We claim that $Q$ is a viscosity
solution of
\begin{equation}\label{eq42}
 \bar{F}(D^2Q)=\fint_{Q_1}f(x)\,dx
\end{equation}
in $\mathbb{R}^n$.
We first prove that $Q$ is a  viscosity subsolution of (\ref{eq42}).
 Fix $\phi\in C^2(\mathbb{R}^n)$ and suppose $Q-\phi$ has a strict local maximum at $x_0$ with $Q(x_0)=\phi(x_0)$. We intend to prove
 $$\bar{F}(D^2\phi(x_0))-\fint_{Q_1}f(x)\,dx\geq 0.$$
 Suppose, to the contrary, that
 $$\delta:=\bar{F}(D^2\phi(x_0))-\fint_{Q_1}f(x)\,dx<0.$$
 Applying Corollary \ref{th4} to $A=D^2\phi(x_0)$, we let $v\in \mathbb{T}$ be a viscosity solution of
 \begin{equation}\label{eq37}
   F(D^2\phi(x_0)+D^2v)=f(x)+\bar{F}(D^2\phi(x_0))-\fint_{Q_1}f(x)\,dx.
 \end{equation}
 Introduce the perturbed test function
 $$\phi^{R_i}(x)=\phi(x)+\frac{1}{R_i^2}v\left(R_i x\right),\ \ \ x\in \mathbb{R}^n.$$

 We claim that
 $$F\left(D^2\phi^{R_i}(x)\right)-f\left(R_i x\right)\leq \frac{\delta}{2},\ \ \ x\in B_r(x_0)$$
 in the viscosity sense for some sufficiently small $r>0$. To see this, we fix $\psi\in C^{\infty}(\mathbb{R}^n)$ such that $\phi^{R_i}-\psi$ has a minimum at a point $x_1\in B_r(x_0)$ with
 $$\phi^{R_i}(x_1)=\psi(x_1).$$
 Then it is obvious that
$$\eta(y):=v(y)-R_i^2\left(\psi\left(\frac{1}{R_i} y\right)-\phi\left(\frac{1}{R_i} y\right)\right)$$
has a minimum at $y_1=R_ix_1$. Furthermore, observing that $v$ is a viscosity solution of (\ref{eq37}), we have
$$F(D^2\phi(x_0)+D^2\psi(x_1)-D^2\phi(x_1))-f\left(R_ix_1\right)\leq \delta,$$
which implies that
$$F(D^2\psi(x_1))-f\left(R_ix_1\right)\leq \frac{\delta}{2}$$
for small enough $r>0$. The claim is proved.

In view of
\begin{equation*}
   \left\{
\begin{aligned}
   &F(D^2\phi^{R_i}(x))-f\left(R_ix\right)\leq \frac{\delta}{2},\ \ \ x\in B_r(x_0),\\
&F(D^2u_{R_i}(x))-f\left(R_ix\right)=0, \ \ \ x\in B_r(x_0),
\end{aligned}
\right.
 \end{equation*}
 the comparison principle for viscosity solutions leads to
 $$(u_{R_i}-\phi^{R_i})(x_0)\leq \max_{\partial B_r(x_0)}(u_{R_i}-\phi^{R_i}).$$
 In addition, letting $R_i \rightarrow \infty$, we obtain
 $$(Q-\phi)(x_0)\leq \max_{\partial B_r(x_0)}(Q-\phi).$$
But, since $Q-\phi$ has a strict local maximum at $x_0$,  we obtain a contradiction.
In the same manner, we can show that $Q$ is a  viscosity supersolution of (\ref{eq42}).

From the Evans-Krylov theorem and the properties in Lemma \ref{le8}, it follows that there exists some $A\in \mathcal{S}^{n\times n}$ such that
$$Q(x)=\frac{1}{2}x^{T}Ax.$$
\end{proof}
\bigskip
We recall  that
$$E=\{k_1 e_1+\cdots +k_n e_n: k_1,\ldots, k_n\ \ \mbox{are integers}\},$$
and the second order difference quotients
$$\Delta_{e}^2u(x)=\frac{u(x+e)+u(x-e)-2u(x)}{\|e\|^2}\ \ \mbox{for}\ \ e\in E.$$
\bigskip
\begin{lem}\label{le3}
For all $e\in E$, we have
$$F_{ij}(D^2u(x))D_{ij}(u(x+e)+u(x-e)-2u(x))\geq 0, \ \ \ \mbox{a.e.}\ \ \  x\in \mathbb{R}^n,$$
where $F_{ij}(D^2u(x))$  is the subdifferential of $F$ at $D^2u(x)$.
\end{lem}
\begin{proof}
Duo to \cite[Theorem 7.1]{CC95},  $u\in W^{2,p}$ for $p>n$.
By the concavity of $F$, we have
\begin{equation*}
  \begin{split}
    F(D^2u(x+e)) & \leq F(D^2u(x))+F_{ij}(D^2u(x))D_{ij}(u(x+e)-u(x)), \\
      F(D^2u(x-e)) & \leq F(D^2u(x))+F_{ij}(D^2u(x))D_{ij}(u(x-e)-u(x)).
  \end{split}
\end{equation*}
The result follows immediately from the periodicity of $f$.
\end{proof}
\bigskip
\begin{lem}\label{le4}
For all $e\in E$, we have
$$\sup_{x\in \mathbb{R}^n}\Delta_{e}^2u(x)<\infty.$$
\end{lem}
\begin{proof}
By the $W^{2,p}$ estimate $(p>n)$ (see \cite[Theorem 7.1]{CC95}), we obtain
\begin{equation}\label{eq25}
  r^{2-\frac{n}{p}}\|D^2u\|_{L^p(B_r)}\leq C (\|u\|_{L^{\infty}(B_{2r})}+r^{2-\frac{n}{p}}\|f\|_{L^p(B_{2r})})
\end{equation}
 for some constant $C>0$ depending only on $n$, $p$, $\lambda$ and $\Lambda$.

 Since we can write
 $\Delta_{e}^2u(x)=\int_{-1}^1\frac{e^{T}}{\|e\|} D^2u(x+te)\frac{e}{\|e\|}(1-|t|)\,dt$, then the quadratic growth (\ref{eq36}) and (\ref{eq25}) give
 \begin{equation*}
   \begin{split}
     \int_{B_r}|\Delta_{e}^2u(x)|^p\,dx & = \int_{B_r}\left|\int_{-1}^1\frac{e^{T}}{\|e\|} D^2u(x+te)\frac{e}{\|e\|}(1-|t|)\,dt\right|^p\,dx\\
       & \leq 2^{p-1}\int_{B_r} \int_{-1}^1\left|\frac{e^{T}}{\|e\|} D^2u(x+te)\frac{e}{\|e\|}(1-|t|)\right|^p\,dt\,dx\\
       &=2^{p-1}\int_{-1}^1  \int_{B_r} \left|\frac{e^{T}}{\|e\|} D^2u(x+te)\frac{e}{\|e\|}(1-|t|)\right|^p\,dx\,dt\\
       &\leq 2^{p-1} \int_{-1}^1  \int_{B_{r+\|e\|}} \left\| D^2u(x)\right\|^p\,dx\,dt\\
       &\leq C(r+\|e\|)^n(1+\|f\|_{L^{\infty}(\mathbb{R}^n)})^p.
   \end{split}
 \end{equation*}
 By Lemma \ref{le3} and the local maximum principle, this gives rise to a pointwise estimate
 \begin{equation*}
\begin{split}
 \sup_{x\in B_r} \Delta_{e}^2u(x)&\leq  C\left(\frac{1}{|B_{2r}|}\int_{B_{2r}}\left|\Delta_{e}^2u(x)\right|^p\,dx \right)^{\frac{1}{p}}  \\
    & \leq C\left(1+\frac{\|e\|}{r}\right)^\frac{n}{p}(1+\|f\|_{L^{\infty}(\mathbb{R}^n)})
\end{split}
\end{equation*}
for some constant $C>0$.
\end{proof}
\bigskip
We will adopt  Caffarelli and Li's \cite{CL04} arguments to carry out the rest proof of the case $n\geq3$.
\bigskip
\begin{lem}\label{le6}
For all $e\in E$, we have
$$\sup_{x\in \mathbb{R}^n}\Delta_{e}^2u(x)=\frac{e^{T}Ae}{\|e\|^2}.$$
\end{lem}
\bigskip
To proceed, we choose $b\in \mathbb{R}^n$ such that
$$w(e_k)=w(-e_k), \ k=1,\ldots,n, $$
where $$w(x):=u(x)-\frac{1}{2}x^{T}Ax-b\cdot x.$$
Since $\bar{F}(A)=\fint_{Q_1}f(x)\,dx$, there exists some $v\in \mathbb{T}$ satisfying
$$F(A+D^2v)=f$$
in the  viscosity sense. In combination  with $F(A+D^2w)=f$, $h:=w-v$ belongs to  the solution set $S(\lambda,\Lambda,0)$ (see \cite{CC95} for the definition). Consequently, to prove that $h$ is constant, by the harnack inequality, it remains to show that $h$ is bounded from above.
\bigskip
\begin{lem}\label{le7}
$$\sup_{\mathbb{R}^n}h< + \infty.$$
\end{lem}
 \bigskip
 Since proofs of the above  two lemmas are the same as  that of  \cite[Proposition 2.3 and Lemma 2.9]{CL04},  the detailed proofs are omitted. As a consequence, we complete the proof of the case $n\geq 3$.

 To conclude the proof of Theorem \ref{th2}, we give a simpler proof for the case $n=2$ without the concavity condition for $F$.
 \bigskip
\begin{proof}[Proof of the case $n=2$]
By \cite[Theorem 7.1]{CC95}, we know $u\in W^{2,p}$ for $p>2$.
Applying the $C^{1,\alpha}$ estimate $(0<\alpha<1)$ in $B_r (r>1)$, we obtain
$$r\|Du\|_{L^{\infty}(B_r)}\leq C(\|u\|_{L^{\infty}(B_{2r})}+r^2\|f\|_{L^{\infty}(B_{2r})}),$$
which together with the quadratic growth (\ref{eq36}) yields
\begin{equation}\label{eq22}
  |Du(x)| \leq C (1+|x|),\ \ \ x\in \mathbb{R}^2
\end{equation}
for some constant $C>0$.

Considering
 $$v_k(x)=u(x+e_k)-u(x),\ \ \ k=1,2,$$
we easily see that $v_k$ satisfies a linear elliptic equation
 $$a_{ij}(x)D_{ij}v_k(x)=0,\ \ \ \mbox{a.e.} \ \ x\in \mathbb{R}^2,$$
 where $a_{ij}$ satisfies the elliptic condition with  ellipticity  constants $0<\lambda \leq \Lambda < \infty$.
Owing to (\ref{eq22}), it is obviously to see  that $v_k$ has  linear growth, i.e.,
 $$|v_k(x)|\leq C (1+|x|) ,\ \ \ x\in \mathbb{R}^2$$
 for some constant $C>0$.

From the classical $C^{1,\alpha}$ interior estimate in two dimension due to Morrey \cite{CBM} and Nirenberg \cite{LN}, it follows that $v_k$ must be a linear function.  Hence  $D^2u(x+e_k)=D^2u(x)$, which implies that
$h(x):=\Delta u(x)$ is periodic. Now we choose some $\tilde{A}\in \mathcal{S}^{2\times 2}$ satisfying $tr(\tilde{A})=\fint_{Q_1}h(x)\,dx$.
Then, by Lax-Milgram theorem, there exists some periodic function $v\in W^{2,2}(\mathbb{R}^2)$ such that
$$\Delta v(x)=h(x)-\fint_{Q_1}h(x)\,dx,\ \ \ \mbox{a.e.} \ \ x\in \mathbb{R}^2.$$
Hence $w(x)=\frac{1}{2}x^{T}\tilde{A}x+v(x)$ satisfies $$\Delta w(x)=h(x),\ \ \ \mbox{a.e.} \ \ x\in \mathbb{R}^2.$$
 It follows that
$$\Delta(u-w)=0 \ \ \ \mbox{in}\ \ \mathbb{R}^2,$$
which yields
$$u-w=\frac{1}{2}x^{T}Ax+b\cdot x+c$$
for some $A\in \mathcal{S}^{2\times 2}$, $b\in \mathbb{R}^2$ and $c\in \mathbb{R}$ since $u-w$ has quadratic growth. Clearly, we obtain the desired result.
\end{proof}
\bigskip
\begin{rmk}\label{rej3}
As mentioned in Remark \ref{rej4}, the Liouville type result also holds for the linear elliptic equation of non-divergence form. That is, if the entire solution $u$ of $a_{ij}(x)D_{ij}u(x)=f(x)$ has  quadratic growth, then $u$ must be a quadratic
polynomial up to a periodic function. Actually, the second order difference  $\Delta_{k}^2u(x)=u(x+2e_k)+u(x)-2u(x+e_k)$ $(k=1,\ldots,n)$ satisfies $a_{ij}(x)D_{ij}\Delta_{k}^2u(x)=0$ and is bounded.  Hence by the Harnack inequality, $\Delta_{k}^2u$ is a constant, which yields that $u$ is a second order polynomial plus a periodic function clearly.

\end{rmk}
\bigskip
\section{Uniformly elliptic equations on exterior domains}
\noindent

In this section, we establish  the existence and Liouville type results for   quadratic  growth
solutions of uniformly elliptic equations with the periodic data on exterior domains.
 Now we give a proof of Theorem \ref{th10}.
 \bigskip
 \begin{proof}[Proof of Theorem \ref{th10}]
 Without loss of generality, we assume that $\Omega$ contains the origin. Let $r_0=\mbox{daim}\ \Omega+1$.  In view of Theorem \ref{th1}, let $v\in \mathbb{T}$ be a viscosity solution of $F(A+D^2v)=f$ in $\mathbb{R}^n$. We set $w:=\frac{1}{2}x^{T}Ax+b\cdot x+v(x)$ and $\bar{C}:=\|w-\varphi\|_{L^{\infty}(\partial \Omega)}$. For $r>r_0$, let $u_r$ be a  viscosity solution of
\begin{equation*}
   \left\{
\begin{aligned}
   &F(D^2u_r)=f\ \ \mbox{in}\ \ B_r \backslash \overline{\Omega},\\
&u_r=\varphi\ \  \mbox{on}\ \ \partial \Omega,\\
&u_r=w \ \  \mbox{on}\ \ \partial B_r.
\end{aligned}
\right.
 \end{equation*}
 Clearly, we see that
 \begin{equation*}
   \left\{
\begin{aligned}
   &F(D^2w)=f\ \ \mbox{in}\ \ B_r \backslash \overline{\Omega},\\
&w+\bar{C} \geq \varphi\ \  \mbox{on}\ \ \partial \Omega,\\
&w-\bar{C} \leq \varphi\ \  \mbox{on}\ \ \partial \Omega.
\end{aligned}
\right.
 \end{equation*}
 Thus, applying the comparison principle for  viscosity solutions yields
 $$w-\bar{C}\leq u_r\leq w+\bar{C} \ \ \mbox{in}\ \ B_r \backslash \overline{\Omega}.$$
Hence we can apply the H\"{o}lder estimate  to $F(D^2u_r)=f$ in any compact subset $K$ of $\mathbb{R}^n \backslash \overline{\Omega}$ to obtain
$$ \|u_r\|_{C^\alpha(K)}
    \leq  C$$
for some constant $C$ independent of $r$. It follows that there exists a function $u\in C(\mathbb{R}^n \backslash \overline{\Omega})$ and a subsequence of $\{u_r\}_{r=1}^\infty$ that converges uniformly to $u$ in compact sets  of $\mathbb{R}^n \backslash \overline{\Omega}$. Moreover, we have
$$F(D^2u)=f\ \ \mbox{in}\ \ \ \mathbb{R}^n \backslash \overline{\Omega}$$
in the viscosity sense with
$$w-\bar{C}\leq u\leq w+\bar{C} \ \ \mbox{in}\ \ \mathbb{R}^n \backslash \overline{\Omega}.$$

We are now in a position to show that $u$ is continuous up to $\partial \Omega$ and coincides with $\varphi$ on $\partial \Omega$. More precisely, for any $x_0\in \partial \Omega$, then we have
$$\lim_{x\in \mathbb{R}^n \backslash \overline{\Omega},\ x\rightarrow x_0}u(x)=\varphi(x_0).$$
To show this, for arbitrary $\epsilon>0$ and fixed $x_0\in \partial \Omega$, by virtue of the continuity of $\varphi$ on $\partial \Omega$, there exists a constant  $0<\delta<1$ such that $|\varphi(x)-\varphi(x_0)|<\epsilon$ if $x \in \partial \Omega\cap B_\delta(x_0)$.
We then define functions $\varphi^{\pm}\in C^{2}(\overline{B_{r_0}\backslash \Omega})$ by
$$\varphi^{\pm}(x)=\varphi(x_0)\pm\left(\epsilon+\frac{2\sup_{\partial \Omega}|\varphi|}{\delta^2}|x-x_0|^2\right).$$
Clearly, we have
\begin{equation*}
   \left\{
\begin{aligned}
&\varphi^{-} \leq \varphi\leq \varphi^{+} \ \  \mbox{on}\ \ \partial \Omega,\\
&\varphi^{-} \leq 0\leq \varphi^{+} \ \  \mbox{on}\ \ \partial B_{r_0}.
\end{aligned}
\right.
 \end{equation*}
 Since the domain $\Omega$  satisfies  a uniform interior sphere condition, we let  $B_{\rho}(z_{x_0})\subset \Omega$ and $\overline{B_{\rho}(z_{x_0})}\cap \partial \Omega={x_0}$.
 Now we consider  functions
 $w^{\pm}\in C^{2}(\overline{B_{r_0}\backslash \Omega})$ defined by
 $$w^{\pm}(x)=\pm \widehat{B}\left(e^{-\widehat{A}\rho^2}-e^{-\widehat{A}|x-z_{x_0}|^2}\right)$$
 for some positive constants $\widehat{A}$ and $\widehat{B}$ to be specified later. For $x\in \overline{B_{r_0}\backslash \Omega}$, choosing large enough $\widehat{A}>\frac{n\Lambda}{2\lambda \rho^2}$ and $\widehat{B}>0$, we obtain
\begin{equation*}
   \begin{split}
       F(D^2w^{+}(x)+D^2\varphi^{+}(x))\leq & F(D^2\varphi^{+}(x))+\mathcal{M}^{+}(D^2w^{+}(x)) \\
       \leq & F\left(\frac{4\sup_{\partial \Omega}|\varphi|}{\delta^2}\right)+2e^{-\widehat{A}|x-z_{x_0}|^2}\widehat{B}\widehat{A}(n\Lambda-2\widehat{A}\lambda \rho^2)\\
       \leq & F\left(\frac{4\sup_{\partial \Omega}|\varphi|}{\delta^2}\right)+2e^{-4\widehat{A}r_0^2}\widehat{B}\widehat{A}(n\Lambda-2\widehat{A}\lambda\rho^2)\\
       \leq & \inf_{\mathbb{R}^n}f,
   \end{split}
 \end{equation*}
 \begin{equation*}
   \begin{split}
      F(D^2w^{-}(x)+D^2\varphi^{-}(x))\geq & F(D^2\varphi^{-}(x))+\mathcal{M}^{-}(D^2w^{-}(x)) \\
       \geq & F\left(-\frac{4\sup_{\partial \Omega}|\varphi|}{\delta^2}\right)-2e^{-\widehat{A}|x-z_{x_0}|^2}\widehat{B}\widehat{A}(n\Lambda-2\widehat{A}\lambda\rho^2)\\
       \geq & F\left(-\frac{4\sup_{\partial \Omega}|\varphi|}{\delta^2}\right)-2e^{-4\widehat{A}r_0^2}\widehat{B}\widehat{A}(n\Lambda-2\widehat{A}\lambda \rho^2)\\
       \geq & \sup_{\mathbb{R}^n}f
   \end{split}
 \end{equation*}
 and
$$\inf_{x\in \partial B_{r_0} }\widehat{B}\left(e^{-\widehat{A}\rho^2}-e^{-\widehat{A}|x-z_{x_0}|^2}\right)\geq \widehat{B}\left(e^{-\widehat{A}\rho^2}-e^{-\widehat{A}(\rho+1)^2}\right)\geq \sup_{\partial B_{r_0}}(|w|+\bar{C}).$$
Clearly, it follows that
$$w^{-}+\varphi^{-}\leq u_r\leq w^{+}+\varphi^{+}\ \  \mbox{on}\ \ \partial \Omega\cup \partial B_{r_0}.$$
Consequently, by the aid of the  comparison principle  for viscosity solutions, we have
 $$w^{-}+\varphi^{-}\leq u_r\leq w^{+}+\varphi^{+}\ \ \mbox{in}\  \ B_{r_0}\backslash \overline{\Omega}.$$
 Furthermore, letting $r\rightarrow \infty$ leads to
$$ w^{-}+\varphi^{-}\leq u\leq w^{+}+\varphi^{+}\ \ \mbox{in}\  \ B_{r_0}\backslash \overline{\Omega},$$
 that is,
 $$ |u(x)-\varphi(x_0)|\leq \epsilon+\frac{2\sup_{\partial \Omega}|\varphi|}{\delta^2}|x-x_0|^2+w^{+}(x),$$
 which immediately implies that $u(x)\rightarrow \varphi(x_0)$ as $x\rightarrow x_0$.

Finally, since $u-w\in S(\lambda,\Lambda,0)$  is bounded in $\mathbb{R}^n \backslash \overline{\Omega}$, we apply the Harnack inequality and the comparison principle to conclude that $\lim_{|x|\rightarrow \infty} (u-w)(x)$ exists.
In particular, if $\frac{\Lambda}{\lambda}< n-1$, we can obtain a more refined error estimate. Indeed, applying the comparison principle to the  Pucci's operators (see \cite[Theorem 1.10]{ASS11} or \cite[Lemma 2.5]{LZP}), we obtain the desired estimate (\ref{eq39}).
 \end{proof}
 \bigskip
 We end up this section with the proof Theorem \ref{th9}.
\bigskip
\begin{proof}[Proof of Theorem \ref{th9}]
For $r>2$, let $u_r$  be a viscosity solution of
\begin{equation*}
   \left\{
\begin{aligned}
   &F(D^2u_r)=f\ \ \mbox{in}\ \ B_r,\\
&u_r=u\ \  \mbox{on}\ \ \partial B_r.
\end{aligned}
\right.
 \end{equation*}

 We will show that $u_r$ is  bounded in compact sets  of $\mathbb{R}^n$. For this purpose, we let $\bar{u}\in C^2(\mathbb{R}^n)$ keeping $\bar{u}=u$ outside $B_2$ and
 set $$F(D^2\bar{u})=f+g\ \ \ \mbox{in}\ \ \ \mathbb{R}^n,$$
 where $g$ is H\"{o}lder continuous with support in $B_2$.  We  choose $\bar{C}>0$ such that
 $$\bar{C}\frac{1}{4}(\Lambda+\lambda(n-1))\left(1-\frac{\lambda}{\Lambda}(n-1)\right)2^{\frac{1}{2}\left(1-\frac{\lambda}{\Lambda}(n-1)\right)-2}\leq -\|g\|_{L^{\infty}(B_2)}.$$
 Then we consider $$E(x)=\bar{C}|x|^{\frac{1}{2}\left(1-\frac{\lambda}{\Lambda}(n-1)\right)},\ \ \ x\in \mathbb{R}^n\setminus \{0\}.$$
 Since $\bar{u}-u_r$ is bounded in $B_1$, there exists a constant $0<\varepsilon<1$ such that
 $$|\bar{u}(x)-u_r(x)| \leq E(x),\ \ \ x\in \partial B_\varepsilon \cup \partial B_r.$$
 For $x\in B_r\backslash B_\varepsilon$, we obtain
 $$\mathcal{M}^{+}(D^2E(x))\leq g(x)\leq \mathcal{M}^{+}(D^2\bar{u}(x)-D^2u_r(x))$$
 and
 $$\mathcal{M}^{-}(D^2\bar{u}(x)-D^2u_r(x))\leq g(x)\leq \mathcal{M}^{-}(-D^2E(x))$$
 From the comparison principle for viscosity solutions, it follows that
 $$  \bar{u}(x)-E(x)\leq u_r(x)\leq \bar{u}(x)+E(x),\ \ \ x\in B_r\backslash B_\varepsilon.$$
 Applying the Alexandroff-Bakelman-Pucci estimate to $F(D^2u_r)=f$ in $B_2$, we have
 $$\|u_r\|_{L^{\infty}(B_2)}\leq \|u\|_{L^{\infty}(\partial B_2)}+\|E\|_{L^{\infty}(\partial B_2)}+C\|f\|_{L^{\infty}(B_2)},$$
 where $C>0$ depends only on $n$, $\lambda$ and $\Lambda$.
Hence we prove that $u_r$ is  bounded in compact sets  of $\mathbb{R}^n$.
From the H\"{o}lder estimate, it follows that there exists a subsequence of $\{u_r\}_{r=1}^\infty$ that converges uniformly to a function $w\in C(\mathbb{R}^n)$ in compact sets  of $\mathbb{R}^n$.
Moreover, we have
$$F(D^2w)=f\ \ \mbox{in}\ \ \ \mathbb{R}^n$$
in the viscosity sense with
\begin{equation}\label{eq45}
  |w(x)-u(x)| \leq E(x),\ \ x\in \mathbb{R}^n \backslash \overline{B_2}
\end{equation}
and therefore
$$|w(x)|\leq C(1+|x|^2), \ \ x\in \mathbb{R}^n$$
for some constant $C>0$. In addition, by Theorem \ref{th2}, we conclude that
$$w(x)=\frac{1}{2}x^{T}Ax+b\cdot x+c+v(x)$$
for  some $A\in \mathcal{S}^{n\times n}$ with $\bar{F}(A)=\fint_{Q_1}f(x)\,dx$, $b\in \mathbb{R}^n$, $c\in \mathbb{R}$ and $v\in \mathbb{T}$.

Finally, due to $F(D^2u)=F(D^2w)=f$ in $\mathbb{R}^n \backslash \overline{B_1}$ in the viscosity sense and (\ref{eq45}), $u-w$ belongs to the solution set $S(\lambda,\Lambda,0)$ and is bounded in $\mathbb{R}^n \backslash \overline{B_1}$.  Following  the arguments in the establishment of the estimate (\ref{eq39}), we obtain
$$|u(x)-w(x)-c^*| \leq C|x|^{1-\frac{\lambda}{\Lambda}(n-1)},\ \ \ x\in \mathbb{R}^n \backslash \overline{B_1}$$
for some $c^*\in \mathbb{R}$ and $C>0$.
\end{proof}
\bigskip

\section{Degenerate elliptic equations in the whole space}
\noindent

This section is devoted to proofs of Theorems \ref{th11} and \ref{th8}. Since we intend to  apply the arguments in the proof of the uniformly elliptic case to the degenerate elliptic case, we will illustrate the signification of the assumptions (H1)-(H3) in the course of  proof.
\bigskip
\begin{proof}[Proof of Theorem \ref{th11}]
\emph{Uniqueness.} It is clear  by the strong maximum principle.

\emph{Existence.} Let $X=C^{2,\alpha}(\mathbb{R}^n) \cap \mathbb{T}$ and $Y=C^{ \alpha}(\mathbb{R}^n) \cap \mathbb{T}$ for some small enough $0<\alpha<1$. We  consider the  map $ \mathcal{F} :   X \times [0,1]\longrightarrow Y$ defined by
 $$(v,t)\longmapsto \mathcal{F}(v,t)= F(A+D^2v)-\fint_{Q_1}F(A+D^2v)\,dx-tf,$$
and the set
$$\mathcal{T}:=\{t\in [0,1]: \mathcal{F}(v_t,t)=0 \ \ \mbox{for some}\ \ v_t\in X\ \ \ \mbox{with} \ \ A+D^2v_t\in \Gamma\}.$$

\emph{Step 1}: \textbf{$\mathcal{T}$ is closed.}
The assumptions  (H1) and (H2) allow us to obtain a priori estimates. Indeed, Let $v_t\in X$ denote a solution of
$$ F(A+D^2v_t)-\fint_{Q_1}F(A+D^2v_t)\,dx=tf,\ \ \ 0\leq t \leq 1.$$
From (H1), it follows that there exists a constant $C_{K}>0$ depending only on $K$ such that
$$\|A+D^2v_t\|_{L^{\infty}(Q_1)}\leq C_{K}.$$
Furthermore, since (H2), we  apply the Evans-Krylov  theory (see  \cite{GT83} and \cite[Theorem 8.1]{CC95}) to obtain
$$ \|v_t\|_{C^{2,\alpha} (Q_1)} \leq C \left(\|v_t\|_{L^{\infty}(Q_2)}+\|A\|+\|f\|_{C^{\alpha}(Q_2)}+\left|\fint_{Q_1}F(A+D^2v_t)\,dx\right|\right),$$
where  $C>0$   depends only on $C_{K}$, $\alpha$ and $n$. By the aid of the  ellipticity condition for $F$ and  the interpolation inequality, we proceed to obtain
$$ \|v_t\|_{C^{2,\alpha} (Q_1)}\leq C (\|v_t\|_{L^{\infty}(Q_1)}+\|A\|+|F(A)|+\|f\|_{C^{\alpha}(Q_1)}).$$
Following the same arguments to derive the estimate (\ref{eq4}), we obtain
\begin{equation*}
  \|v_t\|_{L^{\infty}(Q_1)} \leq C \|f\|_{L^{\infty}(Q_1)}
\end{equation*}
for some constant $C>0$ depending only on $C_{K}$ and  $n$. Consequently, we obtain the estimate
\begin{equation}\label{eq43}
  \|v_t\|_{C^{2,\alpha} (Q_1)}\leq C (\|A\|+|F(A)|+\|f\|_{C^{\alpha}(Q_1)})
\end{equation}
for some constant $C>0$ independent of $t$.

Let $\{t_i\}_{i=1}^{\infty}\subset \mathcal{T}$  be a sequence converging to $t$. By the estimate (\ref{eq43}) and $A+D^2v_{t_i}\in K $, we deduce that  a subsequence of $\{v_{t_i}\}_{i=1}^{\infty}$   converges  uniformly in the $C^2$ norm to $v$ with $A+D^2v\in K$. Hence $t\in \mathcal{T}$, that is, $\mathcal{T}$ is closed.

\emph{Step 2}: \textbf{$\mathcal{T}$ is open.} It is clear.
\end{proof}
\bigskip
\begin{proof}[Proof of Theorem \ref{th8}]
The assumption (H3)
will be useful in the determination of $A$, i.e., we can obtain a corresponding conclusion to Lemma \ref{le5}.
\end{proof}
\bigskip
 \section{Applications  to the $k$-Hessian equations}
\noindent

We are now in  a position to prove Theorems \ref{th6} and \ref{th7}. Firstly, we have the following necessary condition for the existence of a periodic solution to (\ref{eq33}).
\bigskip
 \begin{lem}\label{le9}
 For any $A\in \mathcal{S}^{n\times n}$  and $v\in C^{2}(\mathbb{R}^n)\cap \mathbb{T}$ with  $\kappa(A), \kappa(A+D^2v)\in \Gamma_k$, we have
 $$\int_{Q_1}\sigma_{k}(A+D^2v)\,dx=\sigma_{k}(A).$$
 \end{lem}
 \begin{proof}
 We clearly have
 \begin{equation*}
      \int_{Q_1}(\sigma_{k}(A+D^2v)-\sigma_{k}(A))\,dx =\int_{Q_1} \int_{0}^{1} \sigma_k^{ij}(A+tD^2v)D_{ij}v \,dt\,dx,
 \end{equation*}
 where $\sigma_k^{ij}=\frac{\partial \sigma_k}{\partial a_{ij}}$.  Since the Reilly formula (see \cite{Rei73})
 $$D_i(\sigma_k^{ij}(A+tD^2v))=0, \ \ \ j=1,2,\ldots,n,$$
 we obtain
 \begin{equation*}
   \begin{split}
    \int_{Q_1}(\sigma_{k}(A+D^2v)-\sigma_{k}(A))\,dx =& \int_{Q_1} \int_{0}^{1} D_i(\sigma_k^{ij}(A+tD^2v)D_{j}v) \,dt\,dx, \\
       & = \int_{0}^{1}   \int_{Q_1}D_i(\sigma_k^{ij}(A+tD^2v)D_{j}v) \,dx\,dt=0.
   \end{split}
 \end{equation*}
 \end{proof}
 \bigskip
 In order to prove Theorem \ref{th6}, we need the following two propositions.
 \bigskip
 \begin{pro}\label{pro1}
 Let $A$  be a positive definite matrix and  $f\in C^{2}(\mathbb{R}^n)$ be periodic and positive. If  $v\in C^{2}(\mathbb{R}^n)\cap \mathbb{T}$  with $\kappa(A+D^2v)\in \Gamma_k$ is a   solution of
 \begin{equation}\label{eq34}
  \sigma_k(A+D^2v)=f,
 \end{equation}
then we have the estimate
  $$\|v\|_{C^2(Q_1)}\leq C,$$
where $C>0$ depends only on $\|v\|_{L^\infty(Q_1)}$, $\|f\|_{C^2(Q_1)}$, $\inf_{Q_1}f$, $\|A\|$ and $\|A^{-1}\|$.
 \end{pro}
 \begin{proof}
For an invertible  matrix $Q$ satisfying  $Q^{T}Q=A$, we can set $\widetilde{\sigma}_k\left(\widetilde{M}\right)=\sigma_k\left(Q^{T}\widetilde{M}Q\right)$ for any  $\kappa\left(Q^{T}\widetilde{M}Q\right)\in \Gamma_k$ and $y=Qx$ for any $x,y \in \mathbb{R}^n$. Hence $\tilde{v}(y)=v(Q^{-1}y)$ and $\tilde{f}(y)=f(Q^{-1}y)$ satisfy
$$\widetilde{\sigma}_k(I+D^2\tilde{v})=\tilde{f}.$$
Using \cite[Propostions 1.2 and 1.3]{LL90}, we  have the estimate
 $$\|\tilde{v}\|_{C^2(\mathbb{R}^n)}\leq C,$$
 where $C>0$ depends only on $\|\tilde{v}\|_{L^\infty(\mathbb{R}^n)}$, $\|\tilde{f}\|_{C^2(\mathbb{R}^n)}$ and $\inf_{\mathbb{R}^n}\tilde{f}$.
 Consequently, we can easily obtain the desired estimate.
\end{proof}
\bigskip
Under the assumptions of  Proposition \ref{pro1}, we establish the $C^0$ estimate for periodic solutions $v$ of (\ref{eq34}).
\bigskip
\begin{pro}\label{pro2}
There exists a constant $C>0$ depending only on $k$ and $n$ such that
\begin{equation*}
  \|v\|_{L^\infty(Q_1)}\leq C(\|A\|+\|f\|_{C^{0,1}(Q_1)}^{\frac{1}{k}}).
\end{equation*}
\end{pro}
\begin{proof}
To the contrary, we  find a sequence of  $\{v_i\}_{i\geq 1}\subset C^2(\mathbb{R}^n)\cap \mathbb{T}$, positive definite matrices $\{A_i\}_{i\geq 1}$ and periodic and positive functions $\{f_i\}_{i\geq 1}\subset C^2(\mathbb{R}^n)$ satisfying $\kappa(A_i+D^2v_i)\in \Gamma_k$ and

$$\sigma_k(A_i+D^2v_i)=f_i$$
such that
$$i\left(\|A_i\|+\|f_i\|_{C^{0,1}(Q_1)}^{\frac{1}{k}}\right)\leq \|v_i\|_{L^\infty(Q_1)}.$$
Letting $\tilde{v}_i=\frac{v_i}{\|v_i\|_{L^\infty(Q_1)}}$, $\tilde{A}_i=\frac{A_i}{\|v_i\|_{L^\infty(Q_1)}}$ and $\tilde{f}_i=\frac{f_i}{\|v_i\|_{L^\infty(Q_1)}^k}$, we clearly have

$$\sigma_k(\tilde{A}_i+D^2\tilde{v}_i)=\tilde{f}_i.$$
Form the gradient estimate for $\sigma_k$-equations established by Trudinger \cite{Tru97} and Chou-Wang \cite{CW01}, we have
$$\left\|D\left(\tilde{v}_i+\frac{1}{2}x^{T}\tilde{A}_i x\right)\right\|_{L^{\infty}(Q_2)}\leq C,$$
where $C>0$ depends only on $n$, $k$, $\|\tilde{A}_i \|$, $\|\tilde{f}_i\|_{C^{0,1}(Q_1)}$ and $\|\tilde{v}_i\|_{L^\infty(Q_1)}$.
Consequently, we obviously see that, up to subsequence,
$$\tilde{v}_i \rightarrow \tilde{v} \ \ \mbox{in}\ \ C(Q_1),\ \ \ \tilde{f}_i\rightarrow 0 \ \ \mbox{in}\ \ C(Q_1)\ \ \mbox{and}\ \ \tilde{A}_i \rightarrow 0\ \ \   \mbox{as}\ \ i\rightarrow \infty.$$
Then, by the Newton-Maclaurin inequality, $\tilde{v}$ satisfies  the equation
$$\Delta  \tilde{v}\geq 0\ \ \ \mbox{in}\ \ \ \mathbb{R}^n$$
in the viscosity sense
with $\|\tilde{v}\|_{L^\infty(\mathbb{R}^n)}=1$. Hence, by the strong maximum principle for viscosity solutions, $\tilde{v}$ is a constant, namely, $\tilde{v}=0$ since $\tilde{v} \in \mathbb{T}$. Nevertheless, $\|\tilde{v}\|_{L^\infty(\mathbb{R}^n)}=1$ yields a contradiction.
\end{proof}
\bigskip
 Now we commence with the proof of Theorem \ref{th6}.
 \bigskip
 \begin{proof}[Proof of Theorem \ref{th6}] We only need to check that $\sigma_k$ satisfies the assumptions $(H1)$ and $(H2)$. Then the theorem immediately follows from Theorem \ref{th11}. Clearly, $\sigma_k$ satisfies  $(H2)$. It remains to show that $\sigma_k$ satisfies  $(H1)$ by setting  $\Gamma=\{A\in \mathcal{S}^{n\times n}: \kappa(A)\in \Gamma_k\}$ and $\widetilde{\Gamma}=\{A\in \mathcal{S}^{n\times n}: \kappa(A)\in\Gamma_n\}$.

Owing to Lemma \ref{le9}, we  let $v_t\in C^{2}(\mathbb{R}^n)\cap \mathbb{T}$ with $\kappa(A+D^2v_t)\in \Gamma_k$ be a solution of
 $$\sigma_k(A+D^2v_t)=t\left(f-\fint_{Q_1}f(x)\,dx\right)+\sigma_k(A),\ \ \ 0\leq t \leq 1.$$
Observing that $f-\fint_{Q_1}f(x)\,dx+\sigma_k(A)>0$ in $\mathbb{R}^n$, we see that
\begin{equation*}
  \begin{split}
     t\left(f-\fint_{Q_1}f(x)\,dx\right)+\sigma_k(A)=& t\left(f-\fint_{Q_1}f(x)\,dx+\sigma_k(A)\right)+(1-t)\sigma_k(A) \\
      \geq & \min\left\{\left(f-\fint_{Q_1}f(x)\,dx+\sigma_k(A)\right),\sigma_k(A)\right\}\\
      \geq &  \inf_{\mathbb{R}^n}\left(f-\fint_{Q_1}f(x)\,dx+\sigma_k(A)\right)\ \ \mbox{in}\ \ \mathbb{R}^n.
  \end{split}
\end{equation*}
Hence combining this with Propositions \ref{pro1} and \ref{pro2}, we  have the estimate
 $$\|v_t\|_{C^2(\mathbb{R}^n)}\leq C,$$
 where $C>0$ depends only on $n$, $k$, $\|f\|_{C^2(\mathbb{R}^n)}$, $\inf_{\mathbb{R}^n}\left(f-\fint_{Q_1}f(x)\,dx+\sigma_k(A)\right)$, $\|A\|$ and $\|A^{-1}\|$. Moreover, $\kappa(A+D^2v_t)$ lies in a compact set of $\Gamma_k$.
Consequently, $\sigma_k$ satisfies  $(H1)$.
\end{proof}
\bigskip
Let $\Omega$ be a bounded domain in $\mathbb{R}^n$ and $v\in C(\Omega)$. We recall that the \emph{upper contact set} of $v$ is defined by
$$\Gamma_v^{+}(\Omega)=\{y\in \Omega: v(x)\leq v(y)+p\cdot (x-y)\ \ \mbox{for all}\ \ x\in\Omega, \ \ \mbox{for some}\ \ p\in \mathbb{R}^n\}.$$
\bigskip
\begin{proof}[Proof of Theorem \ref{th7}]
Let $\Gamma=\{A\in \mathcal{S}^{n\times n}: \kappa(A)\in \Gamma_2\}$ and $\widetilde{\Gamma}=\{A\in \mathcal{S}^{n\times n}: \kappa(A)\in\Gamma_n\}$.
In view of  Remark \ref{rm2} and Lemma \ref{le9}, it is obviously seen  that $\sigma_2$ has the homogenisation operator $\sigma_2$ restricted in $\widetilde{\Gamma}$. Since any entire convex viscosity solution of $\sigma_2(D^2u)=1$ must be a quadratic polynomial, which follows from \cite{CY10,SY21,MC21}, $\sigma_2$ satisfies (H3). In view of Theorem \ref{th8}, our next goal is to show that $D^2u$ is bounded. We will divide the proof of this aim into two steps.

\emph{Step 1}. We claim that
$$\sup_{e\in E}\sup_{x\in \mathbb{R}^n}\Delta_{e}^2u(x) \leq C$$
for some $C>0$.

For convenience, we write in the form

$$F(D^2u)=(\sigma_2(D^2u))^{\frac{1}{2}}=f^{\frac{1}{2}}\ \ \ \mbox{in}\ \ \mathbb{R}^n$$
and set
$$F_{ij}(D^2u)=\frac{\partial F}{\partial u_{ij}}(D^2u),\ \ \ \sigma_2^{ij}(D^2u)=\frac{\partial \sigma_2}{\partial u_{ij}}(D^2u).$$
By the concavity of $F$ and the periodicity of $f$, we have
$$F_{ij}(D^2u(x))D_{ij}\Delta_{e}^2u(x)\geq 0,\ \ \ x\in \mathbb{R}^n.$$
Consider the function
$$v(x)=\eta(x) \Delta_{e}^2u(x),\ \ \ x\in B_{2r},$$
where $$\eta(x)=\left(1-\frac{|x|^2}{4r^2}\right)^{\beta}$$
and $\beta>2$ is a constant. A straightforward computation  leads to
$$D_i \eta=-\beta\frac{x_i}{2r^2}\eta^{1-\frac{1}{\beta}},\ \ \ D_{ij}\eta=-\delta_{ij}\frac{\beta}{2r^2}\eta^{1-\frac{1}{\beta}}-\beta(\beta-1)\frac{x_i x_j}{4r^4}\eta^{1-\frac{2}{\beta}}$$
and
\begin{equation*}
  \begin{split}
     F_{ij}(D^2u)D_{ij}v= & F_{ij}(D^2u)(\eta D_{ij}\Delta_{e}^2u+2 D_{i}\eta D_j\Delta_{e}^2u+\Delta_{e}^2uD_{ij}\eta)\\
       \geq & -\Lambda_F (2|D\eta||D\Delta_{e}^2u|+|\Delta_{e}^2u||D^2\eta|)\\
       \geq &    -C\Lambda_F \left(\frac{\eta^{1-\frac{1}{\beta}}}{r}|D\Delta_{e}^2u|+\frac{\eta^{1-\frac{1}{\beta}}+\eta^{1-\frac{2}{\beta}}}{r^2}|D\Delta_{e}^2u|\right)\\
\geq & -\frac{C}{r^2\eta^{\frac{2}{\beta}}} \Lambda_F\left(r\eta^{1+\frac{1}{\beta}}|D\Delta_{e}^2u|+v\right) \ \ \ \mbox{in} \ \ B_{2r},
   \end{split}
\end{equation*}
where $\Lambda_F$ is the maximum eigenvalue of the matrix $(F_{ij}(D^2u))$ and $C>0$ depends only on  $\beta$.
Furthermore, using the concavity of $v$ on $\Gamma_v^{+}(B_{2r})$, we  estimate
\begin{equation*}
  \begin{split}
    \eta|D\Delta_{e}^2u|=  |Dv-\Delta_{e}^2uD\eta|& \leq |Dv|+\Delta_{e}^2u|D\eta| \\
      & \leq \frac{v}{2r-|x|}+\Delta_{e}^2u|D\eta|\\
      &\leq \frac{v}{r\eta^{\frac{1}{\beta}}}\left(\frac{1}{2}+\beta\right).
  \end{split}
\end{equation*}
Consequently, in $\Gamma_v^{+}(B_{2r})$ we have
\begin{equation}\label{eq5}
0\leq -F_{ij}(D^2u)D_{ij}v\leq \frac{Cv}{r^2\eta^{\frac{2}{\beta}}} \Lambda_F.
\end{equation}

Next, it aims  to estimate
$$\frac{(\Lambda_F)^n}{\det (F_{ij}(D^2u))}.$$
From $\sigma_k^{ij}(D^2u)u_{ij}=k \sigma_k(D^2u)$ for $1\leq k\leq n$, we know that
$$\sigma_2^{ij}(D^2u)\geq \frac{\sigma_2(D^2u)}{\sigma_1(D^2u)}\sigma_1^{ij}(D^2u)=\frac{f }{\Delta u}I$$
in the sense of matrices.
Using $\sum_{i=1}^n\sigma_k^{ii}(D^2u)=(n-k+1)\sigma_{k-1}(D^2u)$ for $1\leq k\leq n$ and the Newton-Maclaurin inequality, we obtain
$$\sigma_2^{ij}(D^2u)\leq \left(\sum_{i=1}^n\sigma_2^{ii}(D^2u)\right)I=(n-1)\sigma_1(D^2u)I=C(n)(\Delta u)I$$
in the sense of matrices. Hence observing that
 $F_{ij}(D^2u)=\frac{1}{2}(\sigma_2(D^2u))^{-\frac{1}{2}}\sigma_2^{ij}(D^2u)=\frac{1}{2}f^{-\frac{1}{2}}\sigma_2^{ij}(D^2u)$,
  we have
$$\Lambda_F\leq C(n) \frac{1}{2}f^{-\frac{1}{2}}\Delta u$$
and $$\det (F_{ij}(D^2u))=\left(\frac{1}{2}f^{-\frac{1}{2}}\right)^n\det (\sigma_2^{ij}(D^2u))\geq \left(\frac{1}{2}\right)^nf^{\frac{n}{2}}\frac{1}{(\Delta u)^n}.$$
It follows that
$$\frac{(\Lambda_F)^n}{\det (F_{ij}(D^2u))}\leq  \frac{C(n)}{f^{n}}(\Delta u)^{2n}.$$

Inserting this into (\ref{eq5}) yields
$$0\leq \frac{-F_{ij}(D^2u)D_{ij}v}{(\det (F_{ij}(D^2u)))^{\frac{1}{n}}}\leq \frac{Cv}{r^2\eta^{\frac{2}{\beta}}} \frac{\Lambda_F}{(\det (F_{ij}(D^2u)))^{\frac{1}{n}}} \leq C \frac{v}{r^2\eta^{\frac{2}{\beta}}} (\Delta u)^{2},$$
where $C>0$ depends only on $n$, $\beta$ and $\inf_{\mathbb{R}^n} f$.
Using the Alexandrov maximum principle, we have
\begin{equation*}
  \begin{split}
    \sup_{B_{2r}}v & \leq Cr\left(\int_{\Gamma_v^{+}(B_{2r})}\frac{(-F_{ij}(D^2u)D_{ij}v)^n}{\det (F_{ij}(D^2u))}\,dx\right)^{\frac{1}{n}} \\
      & \leq Cr\left(\int_{\Gamma_v^{+}(B_{2r})}
      \frac{v^n}{r^{2n}\eta^{\frac{2n}{\beta}}} (\Delta u)^{2n}\,dx\right)^{\frac{1}{n}}\\
       &\leq \frac{C}{r}\left(\int_{\Gamma_v^{+}(B_{2r})}
      v^{n(1-\frac{2}{\beta})}\frac{v^{\frac{2n}{\beta}}}{\eta^{\frac{2n}{\beta}}} (\Delta u)^{2n}\,dx\right)^{\frac{1}{n}}\\
      &\leq  \frac{C}{r}\left(\sup_{B_{2r}}v\right)^{1-\frac{2}{\beta}} \left(\int_{B_{2r}}
       (\Delta_{e}^2u)^{\frac{2n}{\beta}}(\Delta u)^{2n}\,dx\right)^{\frac{1}{n}}\\
       &\leq \frac{C}{r}\left(\sup_{B_{2r}}v\right)^{1-\frac{2}{\beta}}  \left(\int_{B_{2r}} (\Delta u(x))^{2nq}\,dx\right)^{\frac{1}{nq}}  \left(\int_{B_{2r}} (\Delta_{e}^2u)^{\frac{2n}{\beta}q'}\,dx\right)^{\frac{1}{nq'}},
  \end{split}
\end{equation*}
where $q=1+\frac{1}{\beta}$ and $q'=1+\beta$.
Then using
$$\left(\int_{B_r}|\Delta_{e}^2u|^p\,dx\right)^{\frac{1}{p}}\leq \left(\int_{B_{r+\|e\|}}|\Delta u|^p\,dx\right)^{\frac{1}{p}}$$
for any $r>0$ and $p\geq 1$, we arrive at
$$ \sup_{B_{2r}}v \leq \frac{C}{r}\left(\sup_{B_{2r}}v\right)^{1-\frac{2}{\beta}}   \left(\int_{B_{2r+\|e\|}}|\Delta u|^{2n\left(1+\frac{1}{\beta}\right)}\,dx\right)^{\frac{1}{n}}.$$
Finally we have
$$\sup_{B_{r}}\Delta_{e}^2u \leq  \frac{C}{r^{\frac{\beta}{2}}} \left(\int_{B_{2r+\|e\|}}|\Delta u|^{2n\left(1+\frac{1}{\beta}\right)}\,dx\right)^{\frac{\beta}{2n}}$$
for any $r>0$, where $C>0$ depends only on $n$, $\beta$ and $\inf_{\mathbb{R}^n} f$.
For any $p> 2n$, we can choose $\beta>2$ such that $p>2n\left(1+\frac{1}{\beta}\right)$. Furthermore, by the integral condition (\ref{eq44}) and the H\"{o}lder inequality, we
have
$$\int_{B_{2r+\|e\|}}|\Delta u|^{2n\left(1+\frac{1}{\beta}\right)}\,dx \leq C(2r+\|e\|)^n,$$
and  therefore
$$\sup_{e\in E}\sup_{x\in \mathbb{R}^n}\Delta_{e}^2u(x) \leq C$$
for some constant $C>0$.

\emph{Step 2}. Using the interior $C^2$ estimate from \cite{GQ19}, the proof is complete. Indeed. for $x\in \mathbb{R}^n$, we set
$$\tilde{u}(y)=u(x+y)-u(x)-Du(x)\cdot y, \ \ \ y\in\mathbb{R}^n.$$
Clearly, we have
$$\tilde{u}(0)=0, \ \ \ \tilde{u}(y)\geq 0,\ \ \ y\in\mathbb{R}^n$$
and
$$\sup_{e\in E}\sup_{y\in \mathbb{R}^n}\Delta_{e}^2\tilde{u}(y)=\sup_{e\in E}\sup_{x\in \mathbb{R}^n}\Delta_{e}^2u(x)=\gamma <\infty.$$
Combining this with the convexity of $\tilde{u}$, we obtain
$$\sup_{B_2}\tilde{u}\leq C$$
for some constant $C>0$ depending only on $n$ and $\gamma$.
Since $$\sigma_2(D^2\tilde{u}(y))=f(x+y),$$
we see, by   \cite[Theorem 3]{GQ19}, that
$$|D^2u(x)|=|D^2\tilde{u}(0)|\leq C$$
for some constant $C>0$ depending only on $n$, $\gamma$, $\inf_{\mathbb{R}^n}f$ and $\|f\|_{C^2(\mathbb{R}^n)}$.
\end{proof}

\bigskip
\section*{Acknowledgement}
This work is supported by NSFC 12071365.

\bigskip
\bibliographystyle{elsarticle-num}

\end{document}